\documentclass[reqno, 12pt]{amsart}
\usepackage{amssymb,amscd}
\usepackage{graphicx}
\usepackage{epstopdf}

\theoremstyle{plain}
\newtheorem{theorem}{Theorem}[section]
\newtheorem{lemma}[theorem]{Lemma}

\theoremstyle{definition}

\newtheorem{remark}[theorem]{Remark}

\numberwithin{equation}{section}

\newcommand{\R}{\mathbb{R}}
\newcommand{\Z}{\mathbb{Z}}
\newcommand{\N}{\mathbb{N}}
\newcommand{\Sone}{\mathbb{S}^1}
\newcommand{\B}{\mathcal{B}}
\newcommand{\C}{\mathcal{C}}
\newcommand{\CC}{\mathcal{CC}}
\newcommand{\M}{\mathcal{M}}
\newcommand{\PF}{P_F}
\newcommand{\Pg}{P_g}
\newcommand{\eps}{\varepsilon}
\renewcommand{\phi}{\varphi}
\renewcommand{\setminus}{\smallsetminus}

\title{No semiconjugacy to a map of constant slope}
\author{Micha\l\ Misiurewicz}
\address{Department of Mathematical Sciences, IUPUI, 402 N. Blackford
    Street, Indianapolis, IN 46202}
\email{mmisiure@math.iupui.edu}
\author{Samuel Roth}
\address{Department of Mathematical Sciences, IUPUI, 402 N. Blackford
    Street, Indianapolis, IN 46202}
\email{sjroth@iupui.edu}
\date{February 11, 2014}

\begin{document}
\begin{abstract}
We study countably piecewise continuous, piecewise monotone interval
maps. We establish a necessary and sufficient criterion for the
existence of a nondecreasing semiconjugacy to a map of constant slope
in terms of the existence of an eigenvector of an operator acting on a
space of measures. Then we give sufficient conditions under which this
criterion is not satisfied. Finally, we give examples of maps not
semiconjugate to a map of constant slope via a nondecreasing map. Our
examples are continuous and transitive.
\end{abstract}
\maketitle

\section{Introduction}\label{intro}

The idea that some interval maps should be conjugate or semiconjugate
to maps of constant slope (we use the term ``constant slope'' instead
of the more accurate but clumsy ``constant absolute value of the
slope'') appeared first about 50 years ago. Parry~\cite{Parry0} proved
that continuous, transitive, piecewise monotone (with finite number of
pieces) interval maps are conjugate to maps of constant slope. Later,
Milnor and Thurston~\cite{MilnorThurston} proved an analogous result,
removing the assumption of transitivity, but replacing conjugacy by
semiconjugacy. Another proof of the Milnor-Thurston theorem appeared
in~\cite{ALM}. That proof, after small modifications, can be used for
maps of graphs or for piecewise continuous maps (with finitely many
pieces)~\cite{AM}.

In all cases we require that the (semi)conjugacy is via monotone maps
preserving orientation. This is a natural requirement; if we drop it
then we get a completely different, and less interesting, problem.
Also in all cases the logarithm of the slope is equal to the
topological entropy of the initial map. This is because the same is
true for the constant slope maps (see~\cite{MS}, \cite{ALM}).

A natural question is what can be said if the map is piecewise
monotone, but with \emph{countably} many pieces. When trying to make
such generalizations, one immediately encounters some basic problems.

The first problem is a definition of a countably piecewise monotone
map. For such continuous maps, what should we assume about the set of
turning points (local extrema)? For instance, if we allow the closure
of this set to be a Cantor set, there may be a substantial dynamics on
it, not captured by our considerations. Thus, it is reasonable to
assume that the closure of the set of turning points is countable.

The second problem is, what should the slope be? For countably
piecewise monotone maps of constant slope it is no longer true that
the entropy is the logarithm of the slope. There are obvious
counterexamples, where all points of the interval, except the
endpoints, are moved to the right, so the entropy is zero, but the
slope is larger than 1. Thus, there is no natural choice of the slope
of the map to which our map should be (semi)conjugate.

Recently, Bobok~\cite{Bobok} considered the case of continuous,
Markov, countably piecewise monotone interval maps. He found a
necessary and sufficient condition for the existence of a
nondecreasing semiconjugacy to a map of constant slope in terms
of the existence of an eigenvector for a certain operator. The
operator is given by a countably infinite 0-1 matrix representing
the transitions in the Markov system, and the criterion
asks for a nonnegative eigenvector in the sequence space $\ell^1$.
Bobok described a rich class of examples satisfying this criterion
and proved that for many of these examples the constant slope so
obtained is the exponential of the topological entropy of
the original interval map. However, he did not give any examples
that violate the criterion.

In the present paper, we study the general case of countably piecewise
continuous, piecewise monotone interval maps without any Markov
assumption. We also establish a necessary and sufficient criterion --
analogous to Bobok's -- for the existence of a nondecreasing
semiconjugacy to a map of constant slope. It is given, like
Bobok's criterion, in terms of existence of an eigenvector of some
operator, but the operator acts on measures rather than on sequences.
Then we construct a class of examples which violate that criterion.
Our examples are continuous and transitive and thus have positive
topological entropy.

The paper is organized as follows. In Section~\ref{notation} we
present notation and define objects used in the next sections. In
Section~\ref{sc} we establish the criterion for the semiconjugacy to a
map of constant slope. In Section~\ref{nosc} we produce sufficient
conditions for this criterion being not satisfied. Theorem~\ref{maint}
can be considered the main technical result of the paper.
Section~\ref{lift} contains the main theorem of the paper,
Theorem~\ref{main}, which gives us a large class of continuous
transitive interval maps for which Theorem~\ref{maint} applies.
Finally, in Section~\ref{examp} we give a concrete example of a
one-parameter family of continuous transitive countably piecewise
monotone interval maps that are not semiconjugate to a map of constant
slope via a nondecreasing map. We show that in this family there are
uncountably many Markov and uncountably many non-Markov maps.

\section{Notation and definitions}\label{notation}

Let us introduce the various notations and definitions required
to address the problem more clearly.

If $P$ is a closed, countable subset of $[0,1]$, then a
component of the complement of $P$ will be called a \emph{$P$-basic
interval}, and the set of all $P$-basic intervals will be denoted
$\B(P)$.

We want to consider countably piecewise monotone interval maps, but
not only continuous, but also piecewise continuous. That would mean
that there exists a closed countable set $P\subset[0,1]$ such that our
map is continuous and monotone on each $P$-basic interval. However,
this creates a question: what should be the values of our map at the
points of $P$? If the map is continuous, this is not a problem.
However, in general there is no good answer. Even if we allow two
values at those points (one-sided limits from both sides), $P$ may
have accumulation points, and there is no natural way of extending
our map to those points. Therefore we choose the simplest solution --
we do not define the map at all at the points of $P$. This is not a
new idea; a similar solution is normally used for instance in the
holomorphic dynamics on the complex projective spaces of dimension
larger than 1.

Thus, we define a class $\C$ of maps $f$ for which there exists a
closed, countable set $P\subset[0,1]$, $f:[0,1]\setminus P\to [0,1]$,
and $f$ is continuous and strictly monotone on each $P$-basic
interval. Note that we assume strict monotonicity; while it is
possible to do everything that we do assuming only monotonicity, the
technical details would be much more involved and they would obscure
the ideas.

Similarly as in measure theory where two functions are considered equal
if they differ only on a set of measure zero, we will consider two
elements of $\C$ equal if they are equal on a complement of a closed
countable set. This gives us a possibility of using different sets $P$
for a given map $f\in\C$. Each such set for which $f$ is continuous
and strictly monotone on $P$-basic intervals will be called
$f$-admissible.

\begin{lemma}\label{compos}
A composition of two maps from $\C$ belongs to $\C$.
\end{lemma}

\begin{proof}
We will show that if $f,g\in\C$, the set $P$ is $f$-admissible, and
$Q$ is $g$-admissible, then the set $P\cup f^{-1}(Q)$ is $g\circ
f$-admissible. First observe that the set $(f|_I)^{-1}(Q)$ is
countable for every $P$-basic interval $I$. Moreover, there are only
countably many $P$-basic intervals. Thus, the set $P\cup f^{-1}(Q)$ is
countable.

We may assume that $0,1\in P$. Let $[a,b]$ be the closure of a
$P$-basic interval. Then $(f|_{(a,b)})^{-1}(Q)$ is closed in $(a,b)$.
Since $a,b\in P$, the set $(P\cup f^{-1}(Q))\cap[a,b]$ is closed in
$[a,b]$. Since $P$ is closed in $[0,1]$, this proves that $P\cup
f^{-1}(Q)$ is closed in $[0,1]$.

Let $I$ be a component of the complement of $P\cup f^{-1}(Q)$. Then $I$
is a subset of a $P$-basic interval and $f(I)$ is a subset of a
$Q$-basic interval, so $g\circ f$ is continuous and strictly monotone on
$I$.
\end{proof}

Using this lemma, by induction we get that if $f\in\C$ then $f^n\in\C$
for every natural $n$. That is, we can iterate a map from $\C$ without
leaving this class of maps.

We do not want to abandon continuous maps. Therefore we consider the
class $\CC$ of continuous maps $f:[0,1]\to[0,1]$ for which there
exists a countable closed set $P\subset[0,1]$ such that
$f|_{[0,1]\setminus P}\in\C$. Results for maps from this class will
follow easily from the results for maps from $\C$. In view of
Lemma~\ref{compos}, composition of maps from $\CC$ belongs to $\CC$,
so in particular, iterates of a map from $\CC$ belong to $\CC$.

We will say that a map $f\in\C$ (or $f\in\CC$) has \emph{constant
  slope} $\lambda$, if for some $f$-admissible set $P$, $f$
restricted to each $P$-basic interval is affine with slope of absolute
value $\lambda$. Clearly, this property depends only on the map $f$,
and not on the choice of an $f$-admissible set $P$.

We will say that a nonatomic measure defined on the Borel
$\sigma$-algebra on the interval $[0,1]$ is \emph{strongly
  $\sigma$-finite} if there is a closed countable set $P\subset[0,1]$
such that each $P$-basic interval has finite measure. We denote by
$\M$ the set of all such measures. Observe that $\M$ is closed under
addition and under multiplication by positive real scalars.

Each map $f\in\C$ induces an operator $T_f:\M\to\M$ that acts on a
measure $\mu\in\M$ as follows. Choose an $f$-admissible set $P$.
For each $P$-basic interval $I$, consider the homeomorphism $f|_I
: I \to f(I)$. Pull back the measure $\mu|_{f(I)}$ by $f|_I$ (that is,
push it forward by $(f|_I)^{-1}$) to a measure on $I$. This defines
$T_f \mu$ on the interval $I$:
\begin{equation*}
(T_f \mu)|_I = (f|_I)^*(\mu|_{f(I)}), \quad I \in \B(P).
\end{equation*}
More explicitly,
\[
(T_f \mu)(A)=\sum_{I\in\B(P)} \mu(f(I \cap A))
\]
for all Borel sets $A$. By a common refinement argument, the
definition of $T_f$ depends only on the map $f$, and not on the choice
of an $f$-admissible set $P$. Moreover, as in the proof of
Lemma~\ref{compos}, if for a closed countable set $Q$ each $Q$-basic
interval has measure $\mu$ finite, then each $R$-basic interval, where
$R=P\cup f^{-1}(Q)$, has measure $T_f\mu$ finite. This shows that
indeed $T_f$ maps $\M$ to $\M$.

Note the linearity properties $T_f(\mu + \nu) = T_f\mu + T_f\nu$
and $T_f(\alpha \mu) = \alpha T_f\mu$ for $\mu, \nu \in \M$ and
$\alpha \geq 0$.

Instead of maps of the interval $[0,1]$ into itself, we can consider
maps of the circle or of the real line into itself. We will use for
them the same notation as for the interval maps.

\section{Semiconjugacy}\label{sc}

Although $\M$ is not a true linear space (multiplication by negative
scalars is not permitted), it is nevertheless quite fruitful to
consider eigenvectors for positive eigenvalues. Let us consider the
meaning of an eigenvector in this setting. Fix a map $f\in\C$ and an
$f$-admissible set $P$. The condition $T_f\mu=\lambda\mu$,
$\lambda > 0$, is equivalent to the condition that for all
$P$-basic intervals $I$, $(f|_I)^*(\mu|_{f(I)})=\lambda\mu|_I$. We
will suppress subscripts and write $f^*\mu=\lambda\mu$ when the
context is clear. However, if $f^*\mu=\lambda\mu$ on a $P$-basic interval
$I$, then the Radon-Nikodym derivative $\frac{df^*\mu}{d\mu}$
is identically $\lambda$ on $I$. This Radon-Nikodym derivative is the
measure-theoretic version of the Jacobian of $f$ on $I$ as defined by
Parry~\cite{Parry}. Thus, a measure $\mu\in\M$ satisfying
$T_f\mu=\lambda\mu$ is a measure of constant Jacobian $\lambda$ for
$f$. If $\mu$ is an eigenvector for $T_f$, then for any subinterval
$J$ contained in a single $P$-basic interval $I$, $\mu(f(J)) = \lambda
\mu(J)$. In words, if $\mu$ is an eigenvector for $T_f$, then within
each $P$-basic interval, $f$ uniformly stretches the $\mu$ measures of
intervals by a factor $\lambda$. It is suggestive to note that
constant slope maps have the same property, but with lengths in place
of measures. We will show that this has a deeper meaning. Let us
denote the Lebesgue measure by $m$.

\begin{lemma}\label{Leb}
The map $f\in\C$ has constant slope $\lambda$ if and only if $T_fm =
\lambda m$.
\end{lemma}

\begin{proof}
Assume that $f$ has constant slope $\lambda$ and let $P$ be
$f$-admissible. Then $f$ is affine with slope $\pm\lambda$ on each
$P$-basic interval. Consider an arbitrary $P$-basic interval $I$ and
the restricted map $f|_I:I\to f(I)$. It suffices to prove that $f^*m =
\lambda m$ on this interval. For every subinterval $(a,b)\subset I$ we
have
\begin{equation*}
(f^*m)(a,b)=m(f((a,b)))=|f(b)-f(a)|=\lambda(b-a)=(\lambda m)(a,b).
\end{equation*}
Since the two Borel measures $f^*m$ and $\lambda m$ on $I$ agree on
all open intervals, they are in fact equal.

Conversely, assume that $T_fm = \lambda m$. Let $I$ be a
$P$-basic interval. Then for every subinterval $(a,b)\subset I$ we have
\begin{equation*}
|f(b)-f(a)|=m(f((a,b)))=(f^*m)(a,b)=\lambda m(a,b)=\lambda(b-a).
\end{equation*}
Therefore $f|_I$ is affine with slope
$\pm\lambda$. But $I$ was arbitrary. Therefore $f$ has
constant slope $\lambda$.
\end{proof}

\begin{remark}\label{Leb-R}
With only slight modifications in wording, the proof of
Lemma~\ref{Leb} goes through for maps of the real line into itself.
\end{remark}

\begin{theorem}\label{semiconj}
Let $f \in \C$ and let $\lambda>0$. Then $f$ is semiconjugate via a
nondecreasing map $\phi$ to some map $g\in\C$ of constant slope
$\lambda$ if and only if there exists a probability measure $\mu\in\M$
such that $T_f\mu = \lambda \mu$.
\end{theorem}

\begin{proof}
Assume that $f$ is semiconjugate to a map $g$ of constant slope
$\lambda$ by a nondecreasing map $\phi$. Let $P$ be an $f$-admissible
set; it is clear that then $\phi(P)$ is a $g$-admissible set. By
Lemma~\ref{Leb}, $T_g(m)=\lambda m$. Since $m$ is nonatomic, it can be
pulled back by the nondecreasing map $\phi$ to define a measure $\mu =
\phi^* m$; explicitly, $\mu(A) = m(\phi(A))$ for all Borel sets $A$.
Then $\mu$ is a nonatomic Borel probability measure. Now let $I$ be
any $P$-basic interval. Take restrictions of $f$, $g$, $\phi$, $m$,
and $\mu$ to the appropriate domains in the following commutative
diagram:
\begin{equation}\label{cd}
\begin{CD}
I @>f>> f(I)\\
@V{\phi}VV @VV{\phi}V\\
\phi(I) @>>g> g(\phi(I))
\end{CD}.
\end{equation}
Using these restricted maps and measures, $f^*(\mu)$ may be computed
on $I$ as
\begin{equation*}
f^*(\mu)=f^*(\phi^*(m))=\phi^*(g^*(m))=\phi^*(\lambda m)=\lambda
\phi^*(m)=\lambda\mu.
\end{equation*}
But $I$ was an arbitrary $P$-basic interval. Therefore $T_f\mu =
\lambda \mu$.

Conversely, assume that there exists a probability measure $\mu\in\M$
such that $T_f\mu=\lambda\mu$. Define a map $\phi$ by
$\phi(x)=\mu([0,x])$. This map is continuous, nondecreasing, and maps
$[0,1]$ onto $[0,1]$. To see that $\phi$ induces a well-defined factor
map $g$, suppose that $x_1<x_2$ and $\phi(x_1)=\phi(x_2)$. If the
interval $[x_1,x_2]$ contains a point of $P$, then $\phi(x_i)$ belongs
to $\phi(P)$, which will be a $g$-admissible set, so there is no need
to define $g$ at $\phi(x_i)$. Otherwise, the
interval $[x_1,x_2]$ is contained in a $P$-basic set $I$,
$\mu([x_1,x_2])=0$, and thus
\[
\mu(f([x_1,x_2]))=(T_f \mu)([x_1,x_2])=\lambda\mu([x_1,x_2])=0.
\]
Therefore $\phi(f(x_1))=\phi(f(x_2))$, so we can define $g$ by
$\phi\circ f\circ\phi^{-1}$. By construction, $f$ is semiconjugate to
$g$ by the nondecreasing map $\phi$. Also by construction, $\phi_*\mu$
is the Lebesgue measure $m$. Since the diagram~\eqref{cd} commutes and
$\phi$ is monotone, we see that $\phi(P)$ is a $g$-admissible set. It
remains to consider $T_g(m)$. Using the same restricted maps and
measures as in diagram~\eqref{cd}, we get
\begin{equation*}
g^*(m)=g^*(\phi_*(\mu))=\phi_*(f^*(\mu))=\phi_*(\lambda \mu)=\lambda
\phi_*(\mu)=\lambda m.
\end{equation*}
Since $I$ was arbitrary, this shows that $T_g(m)=\lambda m$. By
Lemma~\ref{Leb}, $g$ has constant slope $\lambda$.
\end{proof}

\begin{remark}\label{se-cont}
If $f\in\CC$, then the only way we could get a discontinuity of $g$ in
the above construction was when $x_1<x_2$, $\phi(x_1)=\phi(x_2)$, and
there is a point of $P$ between $x_1$ and $x_2$. However, then
\begin{multline*}
\mu(f([x_1,x_2])) = \mu(\bigcup_{I\in\mathcal{B}(P)} f(I\cap [x_1,x_2]))
\leq \sum_{I \in \mathcal{B}(P)} \mu(f(I \cap [x_1,x_2] )) = \\
= (T_f\mu)([x_1,x_2]) = \lambda \mu([x_1,x_2]) = 0
\end{multline*}
By the continuity of $f$, the set $f([x_1,x_2])$ includes the interval
with endpoints $f(x_1)$, $f(x_2)$. Therefore $\phi(f(x_1))=\phi(f(x_2))$,
so no discontinuity is created. This shows that Theorem~\ref{semiconj}
holds with $\C$ replaced by $\CC$, for both $f$ and $g$.
\end{remark}

\begin{remark}\label{se-circle}
With only slight modifications in wording, the proof of
Theorem~\ref{semiconj} and the considerations in Remark~\ref{se-cont}
go through for circle maps.
\end{remark}

\section{No semiconjugacy}\label{nosc}

Now we want to find conditions that prevent semiconjugacy to a map of
constant slope. We start with a technical lemma. One of our
assumptions is that $\lambda>2$. For piecewise monotone maps with
finite number of pieces this type of an assumption is usually
circumvented by taking a sufficiently high iterate of the map. If
$\lambda>1$ then for some large $n$ we get $\lambda^n>2$. However,
here we have another assumption, that the measures of $P$-basic
intervals are bounded away from 0, and taking an iterate of a map
could lead to this condition being violated.

\begin{lemma}\label{long}
Let $f\in\C$. Suppose that there exist $\lambda > 2$, $\delta > 0$,
$\mu \in \M$ and an $f$-admissible set $P$ such that
$T_f\mu=\lambda\mu$ and the measure of every $P$-basic interval $I$ satisfies
$\delta\leq\mu(I)<\infty$. Then for $\mu$ almost every $x$ in $[0,1]$
there exist infinitely many times $n_1 < n_2 < \ldots$ such that $x$
belongs to an interval which is mapped monotonically by $f^{n_k}$ to
an interval of $\mu$-measure at least $\delta$.
\end{lemma}

\begin{proof}
Fix an arbitrarily large natural number $N$ and choose an arbitrary
$P^N$-basic interval $J$. A $P^N$-basic interval means a component of
the complement of $\bigcup_{i=0}^{N-1} f^{-i}(P)$; thus, $f^N$ is
monotone and continuous on each $P^N$-basic interval. It suffices to prove that for
$\mu$ almost every $x$ in $J$ there exists a time $n \geq N$ and a
$P^n$-basic interval $L \subset J$ such that $x\in L$ and $\mu(f^n(L))
\geq \delta$.

If $\mu(f^N(J)) \geq \delta$, then we are done. Otherwise, $J$ is a
``bad'' interval, and we subdivide it at all the points of
intersection $J \cap f^{-N}(P)$ into $P^{N+1}$-basic intervals, which
we classify as either ``good'' or ``bad'' according as
$\mu(f^{N+1}(L))$ is either at least $\delta$ or smaller than
$\delta$, respectively. For points $x$ in the good $P^{N+1}$-basic
intervals, the claim holds. But if any of these intervals $L$ is bad,
we subdivide it further at all the points of intersection $L \cap
f^{-(N+1)}(P)$ into $P^{N+2}$-basic intervals, which we then classify
as good or bad, and so on.

To be more precise, we define $\B_0 = \{J\}$ and we recursively define
\begin{equation*}
\B_{i+1} = \{M\in \B(P^{N+i+1}):\mu(f^{N+i+1}(M)) < \delta
\textnormal{ and } \exists L \in \B_i, M \subset L \}.
\end{equation*}
Now observe that each bad interval $L$ at stage $i$ subdivides into at
most two bad intervals at stage $i+1$, because $\# ( L \cap
f^{-(N+i)}(P) ) = \# ( f^{N+i}(L) \cap P)$, and by hypothesis, an
interval of measure less than $\delta$ never contains more than one
point of $P$. It follows that $\# \B_i \leq 2^i$. The hypothesis
$T_f\mu=\lambda\mu$ means that wherever $f$ is monotone and continuous, it stretches
$\mu$ measures uniformly by the factor $\lambda$. This provides an
upper bound on the measures of bad intervals. If $L \in \B_i$, then
\[
\mu(L)=\lambda^{-(N+i)}\mu(f^{N+i}(L))\leq\lambda^{-(N+i)}\delta.
\]
It follows that
\[
\sum_{L\in\B_i}\mu(L)\leq\left(\frac{2}{\lambda}\right)^i
\lambda^{-N}\delta,
\]
and this quantity tends to zero as $i \to \infty$. Therefore almost
every point of $J$ falls at some stage of the process into a good
basic interval, and this proves our claim.
\end{proof}

The following lemma is an analog of Lebesgue's density theorem. While
it is known, it is difficult to find in the literature the statement
we want. Usually statements with one-sided neighborhoods are only
about the Lebesgue measure, while statements about more general
measures use balls around the density point. Therefore we show how to
deduce what we need from the statement about the Lebesgue
measure.

\begin{lemma}\label{density}
Let $\mu \in \M$ and let $A$ be a Borel set. Then for $\mu$ almost
every $x \in A$ the measures of all one-sided
neighborhoods of $x$ are positive, and
\begin{equation}\label{d0}
\lim_{\delta \searrow 0} \frac{\mu (A \cap [x, x+\delta))}{\mu
  ([x,x+\delta))} = \lim_{\delta \searrow 0} \frac{\mu (A \cap
  (x-\delta, x])}{\mu ((x-\delta,x])} = 1.
\end{equation}
\end{lemma}

\begin{proof}
Let $P$ be a countable closed subset of $[0,1]$ such that the measure
of every $P$-basic interval is finite.
Let $L$ be a $P$-basic interval, and let $a$ denote the left endpoint
of $L$. It suffices to prove the claim for $\mu$ almost every $x$ in
$A \cap L$; therefore we may restrict everything to $L$.

There may exist closed subintervals of $L$ of measure zero. There are
countably many of such maximal intervals, so their union has measure
zero. Hence, we are free to remove this union, as well as the endpoints
of $L$, from $A$. Then, for each $x \in A$ and each positive $\delta$,
the measures $\mu([x,x+\delta))$ and $\mu((x-\delta,x])$ are nonzero.
Since we restrict everything to $L$, and $\mu(L)$ is finite, we see
that the ratios under consideration have finite numerators and
denominators, and so are well-defined.

Introduce a map $\phi:L\to Y$, where $Y=[0,\mu(L)]$, by
$\phi(x)=\mu((a,x))$ By construction, $\phi$ is nondecreasing.
Moreover, $\phi$ is continuous because $\mu$ is nonatomic. By the
definition of $\phi$, every interval $I \subset Y$ enjoys the property
$m(I) = \mu(\phi^{-1}(I))$. It follows that this property holds for
all measurable sets $I \subset Y$. Therefore
\begin{equation}\label{d1}
\lim_{\delta\searrow 0}\frac{\mu(A\cap[x,x+\delta))}
{\mu([x,x+\delta))}=\lim_{\eta\searrow 0}
\frac{m(\phi(A)\cap[\phi(x),\phi(x)+\eta))}
{m([\phi(x),\phi(x)+\eta))}
\end{equation}
and
\begin{equation}\label{d2}
\lim_{\delta\searrow 0}\frac{\mu(A\cap(x-\delta,x])}
{\mu((x-\delta,x])}=\lim_{\eta\searrow 0}
\frac{m(\phi(A)\cap(\phi(x)-\eta,\phi(x)])}
{m((\phi(x)-\eta,\phi(x)])}.
\end{equation}

The preimage under $\phi$ of a set of full Lebesgue measure in $Y$ has
full $\mu$ measure in $L$. By the Lebesgue density theorem (see,
e.g.,~\cite{Faure}), the limits of the right-hand sides of~\eqref{d1}
and~\eqref{d2} are 1 for Lebesgue almost all $x\in\phi(A)$.
Therefore~\eqref{d0} holds for $\mu$ almost all $x\in A$.
\end{proof}

In the next theorem we need an assumption stronger than transitivity.
Since the term ``strong transitivity'' is sometimes used for the
property that the union of the images of every nonempty open set is
the whole space, we have to use another term. We will say that a map
$f\in\C$ is \emph{substantially transitive} if for every nonempty open
set $U\in[0,1]$ the set $[0,1]\setminus\bigcup_{n=0}^\infty f^n(U)$ is
countable.

If $f\in\CC$ is transitive, we get substantial transitivity
automatically. We will need this later also for continuous circle
maps, so we will state a lemma for graph maps. Again, this lemma is
known, but it is easier to prove it than to look for it in the
literature.

\begin{lemma}\label{trans}
Let $X$ be a graph and let $f$ be a topologically transitive
continuous map of $X$ to $X$. Let $U$ be a nonempty open subset of $X$.
Then the set $X\setminus\bigcup_{n=0}^{\infty} f^n(U)$ is finite.
\end{lemma}

\begin{proof}
By replacing $U$ by one of its connected components, we may assume that $U$ is
connected. Let $V$ be the connected component of the set
$W=\bigcup_{n=0}^{\infty} f^n(U)$ that contains $U$. By transitivity,
there is $N$ such that $f^N(V) \cap V \neq \emptyset$. However, $V$ is
a component of a forward invariant set, so $f^N(V) \subset V$. It
follows that $W=\bigcup_{n=0}^{N-1} f^n(V)$, and thus $W$ has only
finitely many connected components. On the other hand, by transitivity
$W$ is dense in $X$. Therefore $W$ excludes only finitely many points
of $X$.
\end{proof}

\begin{remark}\label{trans-R}
If $X=\R$ and $f$ is a lifting of a degree one circle map, the same
proof shows that $\bigcup_{n=0}^{\infty} f^n(U)=\R$. For such $f$
there is a constant $M>0$ such that $|f(x)-x|<M$ for every $x\in\R$,
so $W$ is contained in the set of points whose distance from $V$ is
smaller than $NM$. Since $V$ is connected and $W$ is dense, we must
have $V=\R$.
\end{remark}

Now we can prove the main technical result of the paper.

\begin{theorem}\label{maint}
Let $f\in\C$ be a substantially transitive map and let $\lambda > 2$.
Assume that there exist $\delta > 0$, an infinite measure $\mu\in\M$,
and an $f$-admissible set $P$, such that $T_f\mu=\lambda\mu$ and the
measure of every
$P$-basic interval $I$ satisfies $\delta\leq\mu(I)<\infty$. Then
there is no probability measure $\nu\in\M$ such that
$T_f\nu=\lambda\nu$.
\end{theorem}

\begin{proof}
Suppose that such measure $\nu$ exists. The measure $\mu + \nu$ is an
infinite measure in $\M$ such that
$T_f(\mu+\nu)=\lambda(\mu+\nu)$ and for every $P$-basic interval
$I$ we have $\delta \leq (\mu+\nu)(I) < \infty$. Replace $\mu$ by
$\mu + \nu$ if necessary to obtain absolute continuity of $\nu$ with
respect to $\mu$. Now take the Radon-Nikodym derivative, $\xi =
\frac{d\nu}{d\mu}$ and write $d\nu=\xi d\mu$. Integrate the function $\xi
\circ f$ over any Borel set $A$ contained in any $P$-basic interval
$I$:
\begin{multline*}
\int_A \xi \circ f\;d\mu|_I = \int_{f(A)} \xi \; d f_*(\mu|_I) =
\frac{1}{\lambda}\int_{f(A)} \xi \; d\mu|_{f(I)} =
\frac{1}{\lambda} \int_{f(A)} d\nu|_{f(I)}\\
= \int_{f(A)} df_*(\nu|_I)= \int_A 1 \circ f \; d\nu|_I
=\int_A d\nu|_I= \int_A \xi \; d\mu|_I.
\end{multline*}
This shows that $\xi \circ f = \xi$ $\mu$ almost everywhere; that is,
that up to a set of measure $\mu$ zero, the function $\xi$ is constant
along the orbits of $f$.

By the definition of the Radon-Nikodym derivative, $\int_0^1 \xi \;
d\mu=\nu([0,1])=1$. Therefore there exists a positive real number
$\eps$ such that the measurable set $E = \xi^{-1}([\eps,\infty))$ has
positive $\mu$ measure. This measure cannot be infinite, because then
$\int_0^1 \xi \; d\mu$ would be infinite. Thus, $0 < \mu(E) < \infty$.
Because $\xi$ is constant along orbits, this set $E$ is fully
invariant; that is, $f^{-1}(E) = E$. While this is $\mu$ almost
everywhere, we can modify $E$ by adding/subtracting a set of
$\mu$ measure zero so that it holds everywhere.

The plan of the rest of the proof is as follows. We use
Lemma~\ref{density} to get high density of $E$ in a small interval,
then Lemma~\ref{long} to transport it to a long interval, and then
substantial transitivity to transport it to the whole space. By the
invariance of $E$ this construction gives us infinite measure of $E$,
which is impossible.

There is a point $x\in E$ which satisfies conclusions of both
Lemmas~\ref{long} and~\ref{density} (with $A=E$). In particular, there
exist a sequence of times $n_k$ and a sequence of intervals $L_k =
[a_k, b_k]$ containing $x$ such that $f^{n_k}$ is monotone on
$L_k$ and such that $\mu(f^{n_k}(L_k)) \geq \delta$. Trimming
the intervals $L_k$ (but keeping $x\in L_k$), we can achieve
equality $\mu(f^{n_k}(L_k)) = \delta$. Therefore $\mu(L_k) =
\lambda^{-n_k}\delta$, and this decreases to zero. But every
neighborhood (both two-sided and one-sided) of $x$ has positive
$\mu$ measure. Therefore, $a_k \to x$ and $b_k \to x$ as $k
\to \infty$. Now we may use the fact that $x$ is a density point of
$E$ to conclude that $\mu(E \cap L_k)/\mu(L_k) \to 1$.

Next, we show the density of $E$ in an interval at the large scale. By
compactness, after passing to subsequences we may assume that
$f^{n_k}(a_k)$ converges to some point $a\in[0,1]$ and
$f^{n_k}(b_k)$ converges to some point $b\in[0,1]$. Let $L$ denote
the interval $[a,b]$. We claim that the points $a$ and $b$ are
distinct. Indeed, if $a=b$, then for each $j$ take the interval
$\big(a-\frac{1}{j},a+\frac{1}{j}\big)$. It contains infinitely many
intervals $f^{n_k}(L_k)$ and $\mu(E \cap f^{n_k}(L_k))$ is
approaching $\delta$, so the measure of $E \cap\big(a - \frac{1}{j},a
+ \frac{1}{j}\big)$ is at least $\delta$. Since the measure of $E$ is
finite, we may send $j$ to infinity and find by the continuity of
measure that there is an atom at $a$, which is a contradiction. Thus
$a \neq b$. As $k$ grows, the endpoints $f^{n_k}(a_k),
f^{n_k}(b_k)$ of $f^{n_k}(L_k)$ eventually draw nearer to the
respective endpoints $a, b$ of $L$ than half the distance between $a$
and $b$. Therefore, for sufficiently large $k$, the symmetric
difference $f^{n_k}(L_k) \bigtriangleup L$ consist of two
intervals; one with endpoints $a, f^{n_k}(a_k)$, and the other
with endpoints $b, f^{n_k}(b_k)$. Again by the continuity of
measure, each of these intervals has $\mu$ measure converging to zero.
Therefore $\mu(L \bigtriangleup f^{n_k}(L_k)) \to 0$ as $k \to
\infty$. Together with the invariance of $E$ and the monotonicity of
$f^{n_k}$ on $L_k$, this shows that
\begin{multline*}
\frac{\mu(E \cap L)}{\mu(L)}
\geq \lim_{k\to\infty} \frac{\mu(E \cap f^{n_k}(L_k)) -
  \mu(f^{n_k}(L_k) \bigtriangleup L)}{\mu(f^{n_k}(L_k)) +
  \mu(f^{n_k}(L_k) \bigtriangleup L)}
= \lim_{k\to\infty} \frac{\mu(E \cap f^{n_k}(L_k))}
{\mu(f^{n_k}(L_k))} \\ \\
= \lim_{k\to\infty} \frac{\mu(f^{n_k}( E \cap L_k ))}
{\mu(f^{n_k}(L_k))}
= \lim_{k\to\infty} \frac{\lambda^{n_k} \cdot \mu(E \cap
  L_k)}{\lambda^{n_k} \cdot \mu(L_k)} = 1.
\end{multline*}
Therefore $E \cap L$ has full measure in the interval $L$; that is,
$\mu(L \setminus E) = 0$.

If the invariant set $E$ fills $L$, then it must also fill all the
images $f^n(L)$, $n\in\mathbb{N}$. Indeed, $\mu(f(L)\setminus
E)=\mu(f(L\setminus E))$ by the $f$-invariance of $E$. However,
\[
\mu\big(f(L\setminus E)\big) \leq \sum_{I \in \B(P)} \mu\big(f(I \cap
(L \setminus E))\big) = (T_f\mu)(L \setminus E) = \lambda \mu(L
\setminus E) = 0.
\]
This shows that $\mu(f(L) \setminus E) = 0$, and it follows
inductively that $\mu(f^n(L) \setminus E) = 0$ for all
$n\in\mathbb{N}$.

The interval $L$ has nonempty interior, so by substantial transitivity
of $f$ the set $\bigcup_{n=0}^{\infty} f^n(L)$ excludes at most
countably many points of $[0,1]$, and hence has full $\mu$ measure in
$[0,1]$. But $E$ has full measure in $\bigcup_{n=0}^{\infty} f^n(L)$, and
therefore $E$ has full measure in $[0,1]$. This is a contradiction
because $E$ has finite $\mu$ measure, while by the assumption,
$\mu([0,1])=\infty$. Therefore it is impossible for a probability
measure $\nu\in\M$ to satisfy $T_f\nu=\lambda\nu$.
\end{proof}

\section{Liftings of circle maps}\label{lift}

While Theorem~\ref{maint}, together with Theorem~\ref{semiconj},
provides sufficient conditions for a map $f\in\C$ to have no
semiconjugacy to a map of constant slope, it is not immediately
obvious how to construct concrete examples. In particular, even if we
use those theorems to exclude semiconjugacy to a map with a given
slope $\lambda$, how can we exclude other $\lambda$'s? In this section we
provide tools to do it for a large class of maps. Those maps
additionally will be continuous (formally, they will be restrictions
of continuous maps to $[0,1]\setminus P$). We denote the circle
$\R/\Z$ by $\Sone$.

\begin{theorem}\label{main}
Assume that $f:\Sone\to\Sone$ is a continuous degree one map that is
piecewise monotone with finitely many pieces and has constant
slope $\lambda > 1$. Assume also that $f$ has a lifting $F:\R\to\R$
that is topologically transitive. Take any continuous interval map
$g:[0,1]\to[0,1]$ such that $g|_{(0,1)}$ is topologically conjugate to
$F$. Then there does not exist any nondecreasing semiconjugacy of $g$
to an interval map of constant slope.
\end{theorem}

\begin{proof}
Fix $n$ sufficiently large so that $\lambda^n>2$ and consider the iterates
$g^n$, $f^n$, and $F^n$. The map $g^n|_{(0,1)}$ is conjugate to $F^n$ and $F^n$
is a lifting of the degree one circle map $f^n$ of constant slope $\lambda^n$.
A priori, a transitive map need not have transitive iterates. But $F$ is a
transitive map on the real line, and therefore either all iterates of $F$ are
transitive, or else there exists a point $y\in\R$ such that $F((-\infty,y])=
[y,\infty)$ and $F([y,\infty))=(-\infty,y]$ (see \cite[pgs 156-7]{BlockCoppel}
-- the statements are for interval maps but the proofs also hold in $\R$).  This
latter alternative is impossible for a lifting of a degree one circle map. Therefore
$F^n$ is transitive. If there exists any nondecreasing map that semiconjugates $g$
with a constant slope interval map, then the same map also conjugates $g^n$ with a
constant slope interval map. This shows that after replacing $g$, $f$, and $F$
by some suitably high iterates, we may assume that $\lambda>2$.

Let $h:(0,1)\to\R$ denote the homeomorphism that
conjugates $g|_{(0,1)}$ with $F$. Let $\pi:\R\to\Sone$, denote the
natural projection (that semiconjugates $F$ with $f$). Moreover, let
$s:\R \to \R$, given by $s(x)=x+1$, denote the deck transformation;
then $s \circ F = F \circ s$. Thus, the following diagrams commute.
\begin{equation*}
\begin{CD}
(0,1) @>g>> (0,1)\\
@VhVV @VVhV\\
\R @>>F> \R
\end{CD}
\qquad
\begin{CD}
\R @>F>> \R\\
@V{\pi}VV @VV{\pi}V\\
\Sone @>>f> \Sone
\end{CD}
\qquad
\begin{CD}
\R @>F>> \R\\
@VsVV @VVsV\\
\R @>>F> \R
\end{CD}
\end{equation*}

The hypothesis of piecewise monotonicity means that there
exists a finite set $P \subset \Sone$ such that $f$ is
monotone on each $P$-basic arc, and because of the constant slope, that monotonicity is strict. In the circle, strict monotonicity does not
guarantee injectivity, but after adjoining the finite set $f^{-1}(x)$
for some $x\in\Sone$ to $P$ we have also injectivity of $f$ on
each $P$-basic arc, so that the restriction of $f$ to any $P$-basic
arc is then a homeomorphism onto its image. Just as for interval maps
we define the operator $T_f$ acting on the space of nonatomic, strongly $\sigma$-finite, Borel
measures on the circle. Let $\PF = \pi^{-1}(P)$. Then $\PF$ is a
closed, countable set, invariant under the integer translation map
$s$, and $F$ is strictly monotone on each $\PF$-basic interval. Let
$\Pg = h^{-1}(\PF) \cup \{ 0, 1 \}$. Then $g\in\CC$ and $\Pg$ is a
$g$-admissible set.

Suppose that there is a nondecreasing semiconjugacy of $g$ to a
constant slope interval map, say, with slope $\lambda'$. Then by
Theorem~\ref{semiconj} there exists a probability measure $\nu_g\in\M$
such that $T_g \nu_g = \lambda' \nu_g$. Push this measure down to a
measure $\nu_F = h_*(\nu_g)$ on the real line. Then $h$ gives not only
a topological conjugacy, but also a measure-theoretic isomorphism of
$((0,1),g,\nu_g)$ with $(\R,F,\nu_F)$. It follows that
$T_F\nu_F = \lambda' \nu_F$.

Now push this measure down to the circle, defining $\nu_f = \pi_* \nu_F$.
If $A$ is a Borel subset of a $P$-basic arc in $\Sone$, then its
preimage in the covering space $\R$ can be expressed as a disjoint union
$\pi^{-1}(A)=\bigcup_{n=-\infty}^\infty s^n(B)$ in such a way that $B$ is
a subset of a $P_F$-basic interval. Then for each $n\in\N$, $s^n(B)$
is also a subset of a $P_F$-basic interval, because $P_F$ is
$s$-invariant. By the injectivity of $f$ on each $P$-basic arc, it
follows that the sets $F(s^n(B)), n\in\N,$ are pairwise disjoint. Therefore
$\pi^{-1}(f(A))=\bigcup_{n=-\infty}^\infty F(s^n(B))$ is also a disjoint
union. Now we can calculate
\begin{multline*}
T_f(\nu_f)(A)=\nu_f(f(A))=\nu_F(\pi^{-1}(f(A)))=
\nu_F\left(\bigcup_{n=-\infty}^\infty F(s^n(B))\right)=\\
=\sum_{n=-\infty}^\infty \nu_F(F(s^n(B)))
=\sum_{n=-\infty}^\infty \lambda'\nu_F(s^n(B))
=\lambda'\nu_F\left(\bigcup_{n=-\infty}^\infty s^n(B)\right)=\lambda'\nu_f(A).
\end{multline*}
It follows that $T_f(\nu_f) = \lambda' \nu_f$.

By Theorem~\ref{semiconj} and Remarks~\ref{se-cont}
and~\ref{se-circle}, there is a nondecreasing semiconjugacy of $f$ to
a circle map of constant slope $\lambda'$. But $f$, being a factor of
a transitive map, is itself transitive. Therefore a nondecreasing
semiconjugacy is automatically a conjugacy (see~\cite{ALM}). Thus, $f$
is conjugate to a circle map of constant slope $\lambda'$. The
topological entropy of a constant slope circle map is the logarithm of
the slope (see~\cite{MS}, \cite{ALM}), and topological entropy is a
conjugacy invariant. In such a way
we have shown that if there exists a nondecreasing semiconjugacy of $g$
to an interval map of constant slope $\lambda'$, then in fact
$\lambda'=\lambda$, the constant slope of $f$.

Let us push the Lebesgue measure $m$ on $\R$ via the homeomorphism
$h^{-1}$, and denote $\mu=(h^{-1})_*(m)$. The $\mu$ measures of
$\Pg$-basic intervals are the same as $m$ measures of corresponding
$\PF$-basic intervals. Since $F$ is a lifting of a piecewise monotone
map with finite number of pieces, those measures take only finitely
many values, all of them finite. Moreover $F$, as the lifting of a map
of constant slope, has also constant slope. By Lemma~\ref{Leb} and
Remark~\ref{Leb-R}, the Lebesgue measure $m$ satisfies $T_F m=\lambda
m$. Since $h$ is a measure-theoretic ismorphism of $((0,1),g,\mu)$ with $(\R,F,m)$ it follows that $T_g\mu=\lambda\mu$. 
Substantial transitivity of $g$ follows from transitivity of $F$
and Lemma~\ref{trans}. All this shows that $\mu$ belongs to $\M$ and that $g$ and $\mu$
satisfy the assumptions of Theorem~\ref{maint}. Therefore there is
no probability measure $\nu\in\M$ with $T_g\nu=\lambda\nu$, and
consequently, there is no nondecreasing semiconjugacy of $g$ to a
constant slope interval map.
\end{proof}

\begin{remark}\label{nonvacuous}
In Theorem~\ref{main} there is no difficulty in finding a continuous interval map 
$g:[0,1]\to[0,1]$ such that $g|_{(0,1)}$ is topologically conjugate
to $F$. Let $h$ denote any homeomorphism of $(0,1)$ with $\R$ and define
$g=h^{-1}\circ F\circ h$ with additional fixed points at $0$ and $1$. We
obtain continuity of $g$ at the points $0, 1$, because $F$ was assumed to
be the lifting of a degree one circle map.
\end{remark}

\begin{remark}\label{nontrivial}
There are trivial examples of maps with zero topological entropy for which
there is no semiconjugacy to a map of constant slope, for instance, the map
$f:[0,1]\to[0,1]$ given by $f(x)=x^2$. Theorem~\ref{main} is nontrivial in
that it applies to continuous and transitive interval maps, and by
\cite{BlockCoven} \cite{Blokh}, such maps always have topological entropy at
least $\log\sqrt2$.
\end{remark}

If we wish to construct explicit examples that satisfy the hypotheses
of Theorem~\ref{main}, the only possible difficulty is in verifying the
transitivity of the lifting $F$.  Fortunately, there is a simple
condition, broadly applicable and easy to verify, that
guarantees transitivity.

\begin{theorem}\label{checkTrans}
Assume that $f:\Sone\to\Sone$ is a continuous degree one map that is
piecewise monotone with finitely many pieces and has constant
slope. Assume also that $F:\R\to\R$ is a lifting of $f$.
Let $P$ denote the set of turning points of $F$.  If for each $P$-basic
interval $I$ there are points $x_L,x_R$ in the closure of $I$ such
that $F(x_L)=x_L-1$
and $F(x_R)=x_R+1$, then $F$ is topologically transitive.
\end{theorem}

\begin{proof}
The sets $R:=\left\{x\in\R:F(x)-x=1\right\}$ and $L:=\left\{x\in\R:F(x)-x
=-1\right\}$ are both invariant under integer translations and are both
nonempty by hypothesis. Choose a point $x_L\in L$ and let $x_R$ be the
smallest element of $R$ that is larger than $x_L$. Then $x_R-x_L<1$ and
$F(x_R)-F(x_L)>2$. Since $F$ has constant slope $\lambda$, this shows that
$\lambda > 2$.

Let $U\subset\R$ be any open interval. As $n$ grows, the successive
images $F^n(U)$ grow in length by a factor at least $\lambda/2>1$
until some image $F^N(U)$ contains an entire $P$-basic interval. Within the closure of this
$P$-basic interval there are points $x_L\in L$, $x_R\in R$. Then, in the
next steps, $F^{N+1}(U)$ contains $(x_L-1,x_R+1)$, $F^{N+2}(U)$ contains
$(x_L-2,x_r+2)$, and so on. Therefore the union of all images of $U$ is all
of $\R$, and this proves transitivity of $F$.
\end{proof}

\begin{remark}\label{graph}
We can immediately verify the hypothesis of Theorem~\ref{checkTrans} by
superimposing the diagonal lines $y=x+1$, $y=x-1$ on the graph $y=F(x)$.
Each piece of monotonicity of the graph of $F$ should intersect both diagonal lines.
\end{remark}

\section{Examples}\label{examp}

In this section we provide a concrete example of a one-parameter family of circle
maps of degree one with transitive liftings and constant slope.
In such a way we will have examples where our theorems apply
and there is no nondecreasing semiconjugacy to a map of a constant slope.

Let us describe a lifting $F_\lambda$ in our family. Choose a real parameter
$\lambda\geq2+\sqrt5$. Let $F_\lambda$ be the ``connect the dots'' map (the graph
of $F_\lambda$ consists of straight line segments connecting the dots) with the
dots $(k,k-1)$ and $(k+b,k+c)$, where $k\in\Z$, $b=(\lambda+1)/2\lambda$,
and $c=(\lambda-1)/2$ (see Figure~\ref{fig_ex}). On the interval
$[k,k+b]$ the slope is $(c+1)/b=
\lambda$, and on $[k+b,k+1]$ it is $-c/(1-b)=-\lambda$, so the map has constant
slope $\lambda$. We have
\[
F_\lambda(k+b)-(k+b)-1=\frac{\lambda^2-4\lambda-1}{2\lambda}=
\frac{\left(\lambda-\left(2+\sqrt5\right)\right)
\left(\lambda-\left(2-\sqrt5\right)\right)}{2\lambda}\ge 0,
\]
and therefore $F_\lambda(k+b)-(k+b)\geq1$. Moreover, $F_\lambda(k)-k=-1$, so by the
Intermediate Value Theorem the assumptions of Theorem~\ref{checkTrans}
are satisfied. Thus, $F_\lambda$ is topologically transitive.
Now if we choose any homeomorphism $h:(0,1)\to\R$, we will get a map
$g_\lambda=h^{-1}\circ F_\lambda\circ h$ (with additional fixed points at 0 and 1),
which belongs to $\CC$, but is not semiconjugate by a nondecreasing map to a map of constant
slope. If we want really concrete examples, we can even specify $h$,
for instance $h(x)=\ln(y/(1-y))$ (then $h^{-1}(x)=e^x/(e^x+1)$).

\begin{figure}[ht]
\includegraphics[
height=4in]
{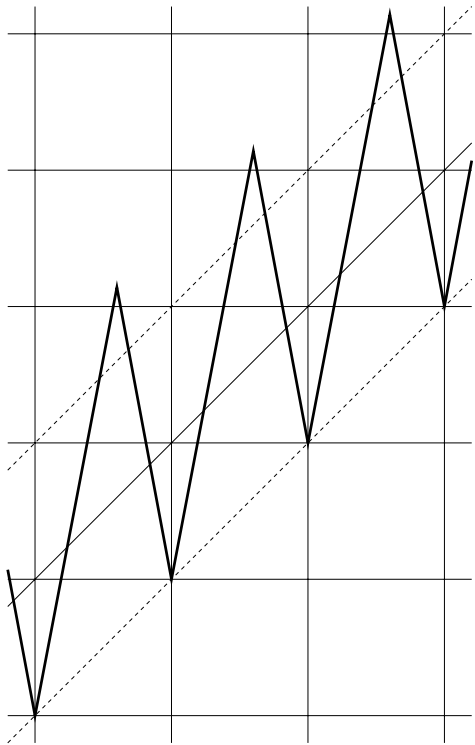}
\caption{The map $F_\lambda$ with $\lambda=5.28$.}\label{fig_ex}
\end{figure}

We would like to have in our family both maps that are and are not
Markov. Remember that ``Markov'' means countably Markov, so
$F_\lambda$ being Markov means that for the corresponding circle map
the trajectory of the local maximum has countable closure (the local
minimum is always a fixed point). Of course $F_\lambda$ is Markov if
and only if the map $g_\lambda$ is Markov.

We start with a lemma on the one-sided full 2-shift
$\sigma:\Sigma\to\Sigma$, where $\Sigma=\{0,1\}^\N$.

\begin{lemma}\label{shift}
Let $D$ be the set of those points $s\in\Sigma$ for which the closure
of the trajectory $\{\sigma^n(s)\}_{n=0}^\infty$ is countable. Then
both sets $D$ and $\Sigma\setminus D$ are uncountable.
\end{lemma}

\begin{proof}
Each element of $\Sigma$ is a 0-1 sequence. Let $E$ be the set of
those sequences that are built of alternating blocks of 0's and 1's,
and the length of the $n$-th block is $n$ or $n+1$. Since we have to
choose between the lengths $n$ and $n+1$ for each $n$, the set $E$ is
uncountable. We claim that $E\subset D$. Fix an element $s\in E$. The
trajectory of $s$ is of course countable. It remains to count the
accumulation points of this trajectory (that is, of the $\omega$-set
of $s$). If we fix the size of a window and slide it sufficiently far
to the right along the sequence $s$, we see in this window only one or
two blocks. This means that every element of the $\omega$-limit set of
$s$ will consist of one or two blocks. However, there are only
countably many such sequences. This proves our claim, and hence $D$ is
uncountable.

The 2-shift is transitive, and therefore the set of points with dense
trajectories contains a dense $G_\delta$ set $G$. If $G$ is countable,
then for every $s\in G$ the set $\Sigma\setminus\{s\}$ is open and
dense, so the intersection of all those sets, $\Sigma\setminus G$, is
a dense $G_\delta$ set. Therefore $(\Sigma\setminus G)\cap G$ is also
a dense $G_\delta$ set, but it is empty. This contradiction shows that
$G$ is uncountable. Since for every $s\in G$ the closure of the
trajectory of $S$ is $\Sigma$, we have $G\subset\Sigma\setminus D$,
and thus $\Sigma\setminus D$ is uncountable.
\end{proof}

\begin{theorem}\label{Markov}
Fix an integer $n\ge 2$. Then there are uncountable sets\break
$\Lambda_M\subset [2n+1,2n+3]$ and $\Lambda_{nM}\subset [2n+1,2n+3]$
such that for every $\lambda\in\Lambda_M$ the lifting $F_\lambda$ is
Markov and for every $\lambda\in\Lambda_{nM}$ the lifting $F_\lambda$
is not Markov.
\end{theorem}

\begin{proof}
Let $A_\lambda$ be the set of those points $x$ such that $F_\lambda^i(x)\in
[0,1]$ for $i=0,1,2,\dots$. Perform the standard coding procedure,
using the left and right subintervals of $[0,1]\cap F_\lambda^{-1}
([0,1])$. It shows that $A_\lambda$ is a Cantor set, and $F_\lambda$
restricted to this set is conjugate to the one-sided full 2-shift. By
the standard argument, for any given itinerary the corresponding point
of $A_\lambda$ depends continuously on $\lambda$. By
Lemma~\ref{shift}, uncountably many itineraries correspond to points
whose trajectories have countable closures, and uncountably many
itineraries correspond to points whose trajectories have uncountable
closures. As $\lambda$ varies from $2n+1$ to $2n+3$ then the image
$-n+(\lambda-1)/2$ under $F_\lambda$ of the local maximum
$-n+(\lambda+1)/(2\lambda)$ sweeps the interval $[0,1]$. When it meets
a point with a countable closure of the trajectory, the corresponding
map $F_\lambda$ is Markov; when it meets a point with an uncountable
closure of the trajectory, it is not Markov. This completes the proof.
\end{proof}

\end{document}